\documentclass{article}

\usepackage{amssymb, amsmath, amsthm}
\usepackage{hyperref}

\theoremstyle{plain}
\newtheorem{thm}{Theorem}
\newtheorem{cor}[thm]{Corollary}
\theoremstyle{definition}
\newtheorem*{rem}{Remark}
\newtheorem*{ack}{Acknowledgments}

\allowdisplaybreaks[4]

\title{A note on the span of Hadamard products of vectors}
\author{Hajime Tanaka \\ \emph{\small Division of Mathematics, Graduate School of Information Sciences, Tohoku University} \\[-1mm] \emph{\small 6-3-09 Aramaki-Aza-Aoba, Aoba-ku, Sendai 980-8579, Japan} \\[-1mm] {\small \emph{E-mail:} htanaka@math.is.tohoku.ac.jp}}
\date{}

\begin{document}

\maketitle

\vspace{-5mm}
\begin{center}
------------------------------------------------------------------------------------------------------------------------------
\end{center}
\vspace{-6mm}
\begin{abstract}
We give a new proof of Theorem 6 in [L. Qiu and X. Zhan, On the span of Hadamard products of vectors, Linear Algebra Appl. 422 (2007) 304--307].

\end{abstract}
\vspace{-5mm}
\begin{center}
------------------------------------------------------------------------------------------------------------------------------
\end{center}

\bigskip

The \emph{Hadamard product} $A\circ B$ of two $m\times n$ complex matrices $A=(A_{ij}), B=(B_{ij})$ is defined by their entry-wise multiplication, i.e., $(A\circ B)_{ij}=A_{ij}B_{ij}$.
Qiu and Zhan \cite{QZ2007LAA} studied the range of the Hadamard product of $n\times n$ complex matrices.
In this note, we give a new proof of the following theorem:

\begin{thm}[{\cite[Theorem 6]{QZ2007LAA}}]\label{thm:Theorem 6}
Let $B_1,B_2,\dots,B_k$ be $n\times n$ complex matrices.
Then
\begin{equation*}
	\mathrm{span}\{(B_1x_1)\circ\dots\circ(B_kx_k)|x_1,\dots,x_k\in\mathbb{C}^n\}=\mathrm{range}((B_1B_1^*)\circ\dots\circ(B_kB_k^*)),
\end{equation*}
where ${}^*$ means conjugate transpose.
\end{thm}

\begin{proof}
Let $e_i$ be the vector in $\mathbb{C}^n$ with a $1$ in the $i$th coordinate and $0$ in all other coordinates.
Then for $1\leqslant i\leqslant n$ the $i$th column of $(B_1B_1^*)\circ\dots\circ(B_kB_k^*)$ is
\begin{equation*}
	((B_1B_1^*)\circ\dots\circ(B_kB_k^*))e_i=(B_1B_1^*e_i)\circ\dots\circ(B_kB_k^*e_i).
\end{equation*}
So
\begin{equation*}
	\mathrm{range}((B_1B_1^*)\circ\dots\circ(B_kB_k^*))\subseteq\mathrm{span}\{(B_1x_1)\circ\dots\circ(B_kx_k)|x_1,\dots,x_k\in\mathbb{C}^n\}.
\end{equation*}
We show
\begin{equation}\label{eqn:inclusion}
	\mathrm{span}\{(B_1x_1)\circ\dots\circ(B_kx_k)|x_1,\dots,x_k\in\mathbb{C}^n\}\subseteq\mathrm{range}((B_1B_1^*)\circ\dots\circ(B_kB_k^*)).
\end{equation}
Pick any $x_1,x_2,\dots,x_k\in\mathbb{C}^n$.
Let $E$ be the orthogonal projection onto $\mathrm{range}((B_1B_1^*)\circ\dots\circ(B_kB_k^*))^{\perp}$.
Then for any $y\in\mathbb{C}^n$ we have
\begin{align*}
	\langle (B_1x_1)\circ\dots\circ(B_kx_k),Ey\rangle&=\sum_{i=1}^n\langle B_1x_1,e_i\rangle\dots\langle B_kx_k,e_i\rangle\overline{\langle Ey,e_i\rangle} \\
	&=\sum_{i=1}^n\langle x_1\otimes\dots\otimes x_k\otimes\overline{y},(B_1^*e_i)\otimes\dots\otimes(B_k^*e_i)\otimes\overline{E}e_i\rangle \\
	&=\bigl\langle x_1\otimes\dots\otimes x_k\otimes\overline{y},\sum_{i=1}^n(B_1^*e_i)\otimes\dots\otimes(B_k^*e_i)\otimes\overline{E}e_i\bigr\rangle,
\end{align*}
where $\overline{\rule{0pt}{1ex}\hspace{1ex}}$ means complex conjugate.
On the other hand, using $E^*=E^2=E$ we have
\begin{align*}
	||\sum_{i=1}^n(B_1^*e_i)\otimes\dots\otimes(B_k^*e_i)\otimes\overline{E}e_i||^2&=\sum_{i,j=1}^n\langle B_1^*e_i,B_1^*e_j\rangle\dots\langle B_k^*e_i,B_k^*e_j\rangle\langle \overline{E}e_i,\overline{E}e_j\rangle \\
	&=\sum_{i,j=1}^n(B_1B_1^*)_{ji}\dots(B_kB_k^*)_{ji}E_{ij} \\
	&=\sum_{i,j=1}^nE_{ij}((B_1B_1^*)\circ\dots\circ(B_kB_k^*))_{ji} \\
	&=\mathrm{trace}(E((B_1B_1^*)\circ\dots\circ(B_kB_k^*))) \\
	&=0.
\end{align*}
So
\begin{equation*}
	\sum_{i=1}^n(B_1^*e_i)\otimes\dots\otimes(B_k^*e_i)\otimes\overline{E}e_i=0,
\end{equation*}
and we find
\begin{equation*}
	\langle (B_1x_1)\circ\dots\circ(B_kx_k),Ey\rangle=0.
\end{equation*}
Since $y$ is arbitrary, we conclude
\begin{equation*}
	(B_1x_1)\circ\dots\circ(B_kx_k)\in\mathrm{range}((B_1B_1^*)\circ\dots\circ(B_kB_k^*)).
\end{equation*}
We have now shown \eqref{eqn:inclusion} and the result follows.
\end{proof}

\begin{cor}[{\cite[Theorem 5]{QZ2007LAA}}]\label{cor:Theorem 5}
Let $A_1,A_2,\dots,A_k$ be $n\times n$ positive semidefinite matrices.
Then
\begin{equation*}
	\mathrm{span}\{(A_1x_1)\circ\dots\circ(A_kx_k)|x_1,\dots,x_k\in\mathbb{C}^n\}=\mathrm{range}(A_1\circ\dots\circ A_k).
\end{equation*}
\end{cor}

\begin{proof}
Set $B_i=A_i^{\frac{1}{2}}$ $(1\leqslant i\leqslant k)$ in Theorem \ref{thm:Theorem 6} and recall $\mathrm{range}(A_i)=\mathrm{range}(B_i)$ $(1\leqslant i\leqslant k)$.
\end{proof}

\begin{cor}[{\cite[Theorem 4]{QZ2007LAA}}]\label{cor:Theorem 4}
Let $A_1,A_2,\dots,A_k$ be $n\times n$ positive semidefinite matrices.
Then
\begin{equation*}
	\mathrm{span}\{(A_1x)\circ\dots\circ(A_kx)|x\in\mathbb{C}^n\}=\mathrm{range}(A_1\circ\dots\circ A_k).
\end{equation*}
\end{cor}

\begin{proof}
For $1\leqslant i\leqslant n$ the $i$th column of $A_1\circ\dots\circ A_k$ is
\begin{equation*}
	(A_1\circ\dots\circ A_k)e_i=(A_1e_i)\circ\dots\circ(A_ke_i).
\end{equation*}
So by Corollary \ref{cor:Theorem 5} we have
\begin{align*}
	\mathrm{range}(A_1\circ\dots\circ A_k)&\subseteq\mathrm{span}\{(A_1x)\circ\dots\circ(A_kx)|x\in\mathbb{C}^n\} \\
	&\subseteq\mathrm{span}\{(A_1x_1)\circ\dots\circ(A_kx_k)|x_1,\dots,x_k\in\mathbb{C}^n\} \\
	&=\mathrm{range}(A_1\circ\dots\circ A_k),
\end{align*}
as desired.
\end{proof}

\begin{rem}
Obviously Corollary \ref{cor:Theorem 5} in turn implies Theorem \ref{thm:Theorem 6} (cf. \cite[Theorem 6]{QZ2007LAA}).
Qiu and Zhan \cite{QZ2007LAA} proved Corollary \ref{cor:Theorem 5} first and got the other theorems.
Their proof makes full use of positive semidefiniteness, and the key lemma there \cite[Lemma 3]{QZ2007LAA} needs a slight analytic argument \cite[p.~15]{Zhan2002B}.
The idea of taking tensor products in this note is from \cite{CGS1978IM}.
See also \cite[Section 2.8]{BI1984B}.
\end{rem}

\begin{ack}
The author would like to thank Akihiro Munemasa and Hiroki Tamura for comments and discussions.
Thanks are also due to Xingzhi Zhan for carefully reading the manuscript.
\end{ack}


\small

\begin{thebibliography}{9}

\bibitem{BI1984B}
E. Bannai and T. Ito,
Algebraic combinatorics I: Association schemes,
Benjamin/Cummings, Menlo Park, CA, 1984.

\bibitem{CGS1978IM}
P. J. Cameron, J.-M. Goethals and J. J. Seidel,
The Krein condition, spherical designs, Norton algebras and permutation groups,
Nederl. Akad. Wetensch. Indag. Math. 40 (1978) 196--206.

\bibitem{QZ2007LAA}
L. Qiu and X. Zhan,
On the span of Hadamard products of vectors,
Linear Algebra Appl. 422 (2007) 304--307.

\bibitem{Zhan2002B}
X. Zhan,
Matrix inequalities,
Lecture Notes in Mathematics, 1790, Springer-Verlag, Berlin, 2002.

\end{thebibliography}
\end{document}